\documentclass[12pt]{amsart}

\usepackage{hyperref, tikz}
\usepackage{amsmath, amssymb, amsfonts}
\usepackage{amsthm}
\usepackage{placeins}
\usepackage{caption}
\usepackage[T1]{fontenc}
\usepackage{youngtab}\begin{scriptsize}
{\footnotesize \begin{small}
{\normalsize \begin{large}

\end{large}}
\end{small}}
\end{scriptsize}
\newtheorem{theorem}{Theorem}[section]

\newtheorem{lemma}[theorem]{Lemma}

\newtheorem{defn}[theorem]{Definition}
\newtheorem{eg}[theorem]{Example}

\newtheorem{claim}{Claim}
\numberwithin{equation}{section}
\title[Asymptotic of Number of Similarity Classes]{Asymptotic of Number of Similarity Classes of Commuting Tuples}
\author{Uday Bhaskar Sharma}
\address{The Institute of Mathematical Sciences, Chennai.}
\email{udaybs@imsc.res.in}
\date{\today}

\newcommand{\F}{\mathbb{F}}

\newcommand{\lbd}{\lambda}

\newcommand{\sbteq}{\subseteq}

\newcommand{\sptq}{\supseteq}
\newcommand{\sptnq}{\supsetneq}
\newcommand{\Max}{\mathrm{max}}

\newcommand{\Mnq}{M_n(\F_q)}
\newcommand{\Gnq}{GL_n(\F_q)}
\begin{document}

\begin{abstract}
Let $c(n, k, q)$ be the number of simultaneous similarity classes of $k$-tuples of commuting $n\times n$ matrices over a finite field of order $q$. We show that, for a fixed $n$ and $q$, $c(n,k, q)$ is asymptotically $q^{m(n)k}$ (upto some constant factor), as a function of $k$, where $m(n) = [n^2/4] + 1$ is the maximal dimension of a commutative subalgebra of the algebra of $n\times n$ matrices over the finite field.
\end{abstract}
\maketitle
\section{Introduction}
Let $\F_q$ be a finite field of order $q$, $n$ be a positive integer, $\Mnq$ be the algebra of $n \times n$ matrices over $\F_q$, and $\Gnq$, the group of invertible $n\times n$ matrices. Then, by the theory of the rational canonical form, the number of similarity classes in $\Mnq$ is given by $$c(n,1,q) = \sum_{\lbd \vdash n} q^{\lbd_1},$$ where $\lbd$ varies over partitions of $n$, and each $\lbd$ is of the form: $$\lbd = (\lbd_1\geq\lbd_2\geq \cdots).$$
It can clearly be seen that, keeping $n$ fixed, $c(n,k,q)$ as a function of $q$ is asymptotically $q^n$ upto multiplication by some constant factor. If we keep $q$ fixed and look at $c(n,1,q)$ as a function of $n$, then also, it is asymptotically $q^n$ upto multiplication by a constant. This is a non-trivial asymptotic result, which Stong \cite{Stong} proved in 1988. In 1995, Neumann and Praeger \cite{NP1} looked at the probability of an $n\times n$ matrix over $\F_q$ being non-cyclic and found that, for a fixed $q$, the probability of a $n\times n$ matrix over $\F_q$ being non-cyclic, is asymptotically $q^{-3}$ as a function of $n$. They also looked at non-separable matrices, and proved that the probability of a matrix in $M_n(\F_q)$ being non-separable is asymptotically $q^{-1}$, upto multiplication by a constant.  In 1998, Girth \cite{Girth} worked on certain probabilities for $n \times n$ upper triangular matrices and compared their asymptotic behaviour with that of corresponding probabilities for arbitrary $n\times n$ matrices over $\F_q$. He also did these comparisons of asymptotic behaviours as $q$ goes to $\infty$, keeping $n$ fixed. The works mentioned above focus mainly on counting in $\Mnq$ and finding the asymptotic behaviours as $n$ goes to $\infty$. \\

In this paper, we shall consider for any positive integer $k$, the space $\Mnq^k$ of $k$-tuples of $n\times n$ matrices over $\F_q$. $\Gnq$ acts on $\Mnq^k$ by simultaneous conjugation, which is defined as follows: $$\text{For $g \in \Gnq$, and $(A_1, \ldots, A_k)\in \Mnq^k$,}$$ $$g.(A_1, \ldots, A_k) = (gA_1g^{-1},gA_2g^{-1},\ldots, gA_kg^{-1}).$$ The orbits for this action are called {\it simultaneous similarity classes}.\\

 Let $a(n,k,q)$ denote the number of simultaneous similarity classes in $\Mnq^k$. Then, by Burnside's lemma we have, $$a(n,k,q) = \frac{1}{|\Gnq|}\sum_{g \in \Gnq}|Z_{\Mnq}(g)|^k,$$ where for each $g \in \Gnq$, $Z_{\Mnq}(g)$ denotes the centralizer algebra of $g$ i.e., $$Z_{\Mnq}(g) = \{ x \in \Mnq\text{ $\mid$ $xg = gx$} \}.$$
\begin{claim}\label{Claim1}
We claim that, keeping $n$ and $q$ fixed, $a(n,k,q)$ is asymptotically $q^{n^2k}$ up to some constant factor, as $k$ goes to $\infty$.
\end{claim}
\begin{proof}
We need to show that there exist positive constants, $m_1$ and $m_2$ (constant with respect to $k$), such that: $m_1q^{n^2k}\leq a(n,k,q) \leq m_2q^{n^2k}$. So, in the Burnside lemma expansion of $a(n,k,q)$, just consider all those $g$ that are scalar matrices. Then, we have $Z_{\Mnq}(g) = \Mnq$. So, taking $m_1$ to be $$m_1 = \frac{q-1}{|\Gnq|},$$ we have $m_1q^{n^2k}\leq a(n,k,q).$\\

Next, if $g$ is a non-scalar matrix, then $Z_{\Mnq}(g) \subsetneq \Mnq$. We know (see Agore \cite{ALAg}), that the maximal dimension of a proper subalgebra of $\Mnq$ is, $n^2 - n + 1$.\\

 So we have \begin{equation*}\begin{aligned} a(n,k,q) &= \frac{1}{|\Gnq|}\sum_{g \in \Gnq}|Z_{\Mnq}(g)|^k\\
&= \frac{1}{|\Gnq|}(q-1)q^{n^2k} + \sum_{ \substack{g \in \Gnq\\g \notin \F_q.I_n} }|Z_{\Mnq}(g)|^k\\
&\leq \frac{1}{|\Gnq|}(q-1)q^{n^2k} + \sum_{ \substack{g \in \Gnq\\g \notin \F_q.I_n}} q^{(n^2-n+1)k} \\
&= \frac{1}{|\Gnq|}(q-1)q^{n^2}k \left( 1 + (|\Gnq| - q + 1)q^{-(n-1)k}\right).
\end{aligned}
\end{equation*}
From this, we get $m_2$ such that $a(n,k,q) \leq m_2q^{n^2k}$. Thus the claim is proved.\end{proof}

Now, denote by $\Mnq^{(k)}$, the set of $k$-tuples of commuting matrices from $\Mnq$, i.e., the set,
\begin{displaymath}
\Mnq^{(k)} = \{ (A_1, \ldots, A_k) \in \Mnq^k\text{ $\mid$ $A_iA_j = A_jA_i$ for $i \neq j$}\}.
\end{displaymath}

Let $c(n,k,q)$ denote the number of simultaneous similarity classes in $\Mnq^{(k)}$ under the simultaneous conjugation by $\Gnq$ on it. The aim of the paper is to find for a fixed $n$ and $q$, an asymptotic for $c(n,k,q)$ as a function of $k$. The problem here is that, the technique used in the proof of Claim~\ref{Claim1} fails in this case because the matrices, $A_1,..., A_k$, are no longer independently chosen.\\

In \cite{UBS}, $c(n,k,q)$ was calculated for $n = 2, 3, 4$. The leading terms of some of those values are shown in Table~\ref{tab1}. \\
\begin{table}[h!]
\begin{equation*}\def\arraystretch{1.2}
\begin{array}{|c|c|c|c |}
\hline
k & c(2,k,q) &  c(3,k,q) & c(4,k,q)\\ [0.3em] \hline
1 & q^2 + q & q^{3} + q^{2} + q & q^4+q^3 + 2q^2 + 2 \\
2 & q^{4} + q^{3} + q^{2}  & q^{6} + q^{5} + 2 q^{4} + \cdots & q^{8} + q^{7} + \cdots \\
\vdots & \vdots  & \vdots & \vdots\\
7 & q^{14} +  q^{13} + \cdots & q^{21} + q^{20} + \cdots & 2q^{28} + q^{27} + \cdots\\
8 & q^{16} + q^{15} +\cdots & q^{24} + q^{23} + \cdots & q^{33} + q^{32} + \cdots\\
\vdots & \vdots & \vdots & \vdots\\
20& q^{40} + q^{39} + \cdots & q^{60} + q^{59} + \cdots & q^{93} + q^{92} + \cdots  \\
21& q^{42} + q^{41} + \cdots & q^{63} + q^{62} + \cdots & q^{98} + q^{97} + \cdots\\
\vdots & \vdots & \vdots & \vdots\\ \hline
\end{array}
\end{equation*}
\caption{Leading terms of $c(n,k,q)$ for $n= 2,3,4$}
\label{tab1}
\end{table}

From Table~\ref{tab1}, we see that $c(2,k,q)$ is asymptotically $q^{2k}$. $c(3,k,q)$ is asymptotically $q^{3k}$ and $c(4,k,q)$ is asymptotically $q^{5k-7} = q^{-7}q^{5k}$. In the case of $n = 4$, we see that $c(4,k,q)$ is asymptotically $q^{5k}$ (and not $q^{4k}$, as we would expect), up to a constant factor which is $q^{-7}$.\\

The number 5 is the maximal dimension for any commutative subalgebra of $M_4(\F_q)$. In fact, Jacobson \cite{NJac} showed that, for any positive integer $n$, the maximal dimension of any commutative subalgebra of $\Mnq$ is $$m(n) = \left[\frac{n^2}{4}\right] + 1.$$

Coming back to $n = 2,3,4$, we see that $m(2) = 2$, $m(3) = 3$ and $m(4) = 5$. So for $n \in\{2,3,4\}$, $c(n,k,q)$ is asymptotically $q^{m(n)k}$ up to some constant factor. We claim this is true for any $n$. Thus we have the main theorem of this paper:
\begin{theorem}\label{main}
For a fixed positive integer $n$ and prime power $q$, $c(n,k,q)$ as a function of $k$, is asymptotic to $q^{m(n)k}$ up to some constant factor.
\end{theorem}
\subsection{Outline of the Paper} In Section~\ref{S2}, we will prove the main theorem (Theorem~\ref{main}). In Section~\ref{S3}, we will find out the asymptotic of counting the total number of $k$-tuples of commuting matrices over $\F_q$ i.e., the cardinality of $\Mnq^{(k)}$.


\section{Proof of Theorem~\ref{main}}\label{S2}
To prove Theorem~\ref{main}, it suffices to prove the existence of positive numbers, $C_1$ and $C_2$, such that: $$C_1q^{m(n)k} \leq c(n,k,q) \leq C_2q^{m(n)k}$$ for large $k$. Before we go ahead, we will need to unravel $c(n,k,q)$.\\

We first define the following:
\begin{defn} Let $Z \sbteq \Mnq$ be a subalgebra, and $Z^*$ be the group of units of $Z$. For positive integer $k$, let $c(Z,k,q)$ denote the number of simultaneous similarity classes of $k$-tuples of commuting matrices in $Z$, under the conjugation action by $Z^*$.\\

For $k = 0$ and any subalgebra $Z\sbteq \Mnq$, $c(Z, 0, q) = 1$.\end{defn}
We claim:
\begin{equation}\label{Eqn1}
c(n,k,q) = \sum_{Z \sbteq \Mnq} c_Z c(Z,k-1,q),
\end{equation}
where $Z$ runs over subalgebras of $\Mnq$, $c_Z$ is the number of similarity classes in $\Mnq$, whose centralizer algebra is conjugate to $Z$. \\

Let $(A_1,\ldots,A_k) \in \Mnq^{(k)}$. Let $Z = Z_{\Mnq}(A_1)$. Then it is clear that $(A_2, \ldots, A_k) \in Z^{(k-1)}$. The map, $$(A_1, \ldots, A_k) \mapsto (A_2, \ldots, A_k),$$ induces a bijection between the set of simultaneous similarity classes in $\Mnq^{(k)}$, which have an element whose first coordinate is $A_1$, and the orbits for the simultaneous conjugation action of $Z^*$ on $Z^{(k-1)}$. Hence we get the identity~(\ref{Eqn1}).\\

Now, in identity~(\ref{Eqn1}), for each $Z$, we can expand $c(Z, k-1, q)$ (when $K \geq 2$) to get
$$c(Z,k-1,q) = \sum_{Z' \sbteq Z}c_{ZZ'}c(Z',k-2,q),$$
where $c_{ZZ'}$ is the number of orbits of matrices in $Z$ for the action of $Z^*$ on it by conjugation, whose centralizer algebra under this conjugation action is conjugate to $Z'$. 

Proceeding this way, we get the following expansion for $c(n,k,q)$:
 \begin{equation}\label{Eqn2}
 c(n,k,q) = \sum_{Z_1 \sptq \cdots \sptq Z_k}c_{Z_1}c_{Z_1Z_2}\cdots c_{Z_{k-1}Z_k},
 \end{equation}
where, for $1 \leq i \leq k-1$, $Z_i$ is the common centralizer of some $i$-tuple of commuting matrices $(A_1,\ldots,A_i)$, i.e., $$Z_i = \bigcap_{j=1}^i Z_{\Mnq}(A_j), $$ and $c_{Z_iZ_{i+1}}$ denotes the number of orbits of matrices in $Z_i$ for the conjugation action of $Z_i^*$, whose centralizer algebra in $Z_i$, is conjugate to $Z_{i+1}$. For $Z_{i+1}\sbteq Z_i$, we say that $Z_{i+1}$ is a \textbf{branch} of $Z_i$, if $c_{Z_iZ_{i+1}} > 0$.\\

Here are some observations about these non-increasing sequences of subalgebras which come up in the expansion of $c(n,k,q)$. We shall state them as a lemma:
\begin{lemma}\label{Obs}
Given a non-increasing sequence of centralizer subalgebras which occurs in equation~(\ref{Eqn2}), say $$Z_1\sptq \cdots \sptq Z_k,$$ we have the following:\begin{enumerate}
\item If for some $i$, $Z_i$ is a commutative subalgebra, then $$Z_{i+1} = \cdots = Z_k = Z_i$$ and for each $j~(i \leq j \leq k-1)$, $$c_{Z_jZ_{j+1}} = q^{\dim(Z_i)}.$$
\item If $Z_i$ is not necessarily commutative, but if $Z_{i+1} = Z_i$, then $$c_{Z_iZ_{i+1}} = q^{\dim(Z(Z_i))},$$ where $Z(Z_i)$ is the centre of $Z_i$.
\end{enumerate}
\end{lemma}
\begin{proof}
For $i \geq 1$, let $(A_1, \ldots, A_i) \in \Mnq^{(k)}$, such that
$$Z_i = \bigcap_{j=1}^iZ_{\Mnq}(A_j).$$
\begin{enumerate}
\item Suppose, for some $i$, $Z_i$ is commutative. Then, for any element, $A_{i+1} \in Z_i$, its centralizer $Z_{Z_i}(A_{i+1})$ in $Z_i$, is $Z_i$ itself. Therefore, we have $$Z_{i+1} = \bigcap_{j=1}^{i+1}Z_{\Mnq}(A_j) = Z_{Z_i}(A_{i+1}) = Z_i,$$  and therefore $c_{Z_iZ_{i+1}} = |Z_i| = q^{\dim(Z_i)}$. Similarly, $Z_j = Z_i$ for $i+1 \leq j \leq k$. Thus, $c_{Z_jZ_{j+1}} = q^{\dim(Z_i)}\leq q^{m(n)}$ for $i \leq j \leq k-1$.

\item If $Z_i$ is not necessarily commutative but, $Z_{i+1} = Z_i$, then $c_{Z_iZ_{i+1}}$ is the number of matrices $A_{i+1}$ in $Z_i$ for which $$Z_{Z_i}(A_{i+1}) = Z_i.$$ Thus $c_{Z_iZ_{i+1}}$ is the size of the centre $Z(Z_i)$, of  $Z_i$. So $$c_{Z_iZ_{i+1}} = q^{\dim(Z(Z_i))} \leq q^{m(n)}.$$
\end{enumerate}
\end{proof}

\subsection{Finding Crude Lower and Upper bounds for $c(n,k,q)$}

The first and main thing we need to show is that there exists a tuple of commuting matrices whose common centralizer is a commutative algebra of dimension $m(n)$. Here are examples of tuples of commuting matrices whose common centralizer is a commutative subalgebra of $\Mnq$ of dimension $m(n)$.

\begin{eg}\label{Eg1}
When $n$ is even, say $n = 2l$, for some $l \geq 1$, we have $$m(n) = l^2 + 1.$$ Consider the commuting tuple, $(A_1, A_2, \ldots, A_{l+1})$, in which
$$A_1 = \begin{pmatrix} 0_l &I_l\\ 0_l & 0_l\end{pmatrix},$$
where $0_l$ is the $l\times l$ 0-block, and $I_l$ is the $l\times l$ identity matrix. For $i \geq 2$, $$A_i = \begin{pmatrix}
0_l & N_i \\ 0_l & 0_l\end{pmatrix},$$ where for $i =2, \ldots, l+1$, $$N_i = \begin{pmatrix}0_{(l-1)\times l}\\e_{i-1} \end{pmatrix} \text{ ( $0_{(l-1)\times l}$ is the $(l-1)\times l$ 0-block)}$$ and $e_{i-1}$ is the $1\times l$ row matrix
$$\begin{pmatrix}0 \cdots \underset{\underset{(i-1)th \text{ place}}\downarrow}{1} \cdots 0\end{pmatrix}.$$
 Its common centralizer algebra is $$Z = \left\{a_0I_n + \begin{pmatrix} 0_l&B\\ 0_l& 0_l\end{pmatrix} \text{ : $a_0 \in \F_q$ and $B \in M_l(\F_q)$ } \right\}.$$
It is commutative and is of dimension $l^2 +1$.
\end{eg}
\begin{eg}\label{Eg2}When $n$ is odd, say $n = 2l +1$ for some $l \geq 1$, then $m(n) = l(l+1) + 1$. Consider the commuting tuple $(A_1, A_2, \ldots, A_{l+1})$ where $$A_1 = \begin{pmatrix} 0_{(l+1 )\times (l+1)} &I_l \\ 0_{l\times (l+1)} & 0_{(l+1) \times l}\end{pmatrix}, $$ and for $i=2, \ldots, l+1$, $$A_i=\begin{pmatrix} 0_{(l+1 )\times (l+1)} &N_i \\ 0_{l\times (l+1)} & 0_{l\times l}\end{pmatrix}, $$ where for each $i$, $N_i$ is a $(l+1)\times l$-matrix of the form $$\begin{pmatrix} 0_{l\times l}\\e_{i-1}\end{pmatrix},$$ where $e_{i-1}$ is as defined in Example~\ref{Eg1}. Then the common centralizer of this tuple of commuting matrices is $$\left\{ a_0I_n + \begin{pmatrix} 0_{(l+1)\times(l+1)} & B \\ 0_{l\times(l+1)} & 0_{l\times l}\end{pmatrix} \text{ : $a_0\in \F_q$ and $B \in M_{(l+1) \times l}(\F_q)$} \right\}.$$ It is commutative and is of dimension $l(l+1) + 1$, which is equal to $m(n)$.
\end{eg}

So we can find at least a $([n/2] + 1)$-tuple of commuting $n \times n$ matrices, whose common centralizer algebra is of dimension $m(n)$.

 \begin{lemma}\label{lower1} There exists $C_1 > 0 $ such that $C_1q^{m(n)k} \leq c(n,k,q)$ for large $k$.
 \end{lemma}
\begin{proof}
Let $l_0 = \left[\displaystyle\frac{n}{2}\right] + 1$.
Consider the $k$-tuple, $$(A_1, A_2, \ldots, A_{l_0}, A_{l_0+1},\ldots, A_k),$$ where the first $l_0$ matrices of the commuting tuple are as in Examples~\ref{Eg1}~or~\ref{Eg2} (depending on whether $n$ is even or odd). Here, $Z_{l_0}$ is a commutative subalgebra of dimension $m(n)$ (as described in the examples). Hence, by Lemma~\ref{Obs}, for $i = l_0+1, \ldots, k$, $Z_i = Z_{l_0}$. Then $$c(n,k,q) \geq  c_{Z_1}c_{Z_1Z_2}\cdots c_{Z_{l_0-1}Z_{l_0}}q^{m(n)(k-l_0)}.$$ Let $$C_1 = \frac{c_{Z_1}c_{Z_1Z_2}\cdots c_{Z_{l_0-1}Z_{l_0}}}{q^{m(n)l_0}}$$ then $c(n,k,q) \geq C_1q^{m(n)k}$ for all large $k$.

\end{proof}

To complete the proof of the Theorem~\ref{main}, we need the following observation (Lemma~\ref{L2}), about the non-increasing sequences of subalgebras, $Z_1 \supseteq \cdots \supseteq Z_k$, which occur in the expansion of $c(n,k,q)$ (given in equation~(\ref{Eqn2})).
\begin{lemma}\label{L2}
$Z(Z_i) \subseteq Z(Z_{i+1})$ for $i \geq 1$ and if $Z_{i+1} \subsetneq Z_i$, then $Z(Z_i) \subsetneq Z(Z_{i+1})$
\end{lemma}
\begin{proof}
Let $x \in Z(Z_i)$. Then, for any $y \in Z_i$ such that $Z_{Z_i}(y) = Z_{i+1}$,  $xy = yx$ implies that $x \in Z_{i+1}$. Now, as $x \in Z(Z_i)$, $xz = zx$ for every $z \in Z_{i+1}$, which implies that, $x \in Z(Z_{i+1})$. So $Z(Z_i) \sbteq Z(Z_{i+1})$ and thus $\dim(Z(Z_{i+1})) \geq \dim(Z(Z_i))$.\\

If $Z_i \supsetneq Z_{i+1}$. Then consider any $y \in Z_i$ for which $Z_{Z_i}(y) = Z_{i+1}$. Clearly, $y \in Z(Z_{i+1})$. But, for $x \notin Z_{i+1}$, $yx \neq xy$. Hence $y \notin Z(Z_i)$. Therefore $Z(Z_i) \subsetneq Z(Z_{i+1})$. Thus $\dim(Z(Z_{i+1})) > \dim(Z(Z_i))$.  \end{proof}

Now we are in a position to get a crude upper bound for $c(n,k,q)$. Let $k > n^2$. Let us look at any summand of $c(n,k,q)$. A summand of $c(n,k,q)$ is of the form, $$c_{Z_1}c_{Z_1Z_2}\cdots c_{Z_{k-1}Z_k},$$ where $Z_1 \supseteq Z_2 \supseteq \cdots \supseteq Z_k$. Let $j$ be the number of distinct $Z_i$'s in the non-increasing sequence. As $\Mnq$ is of dimension $n^2$, there cannot be more than $n^2$ distinct $Z_i$'s in this non-increasing sequence, $Z_1 \supseteq Z_2 \supseteq \cdots \supseteq Z_k$, of subalgebras of $\Mnq$. So $1 \leq j \leq n^2$.\\

We therefore rewrite $c(n,k,q)$ as
\begin{equation}\label{Eqn3} c(n,k,q) = \sum_{j = 0}^{n^2-1} \sum_{\substack{Z_1\sptq\cdots \sptq Z_k \\ j+1~\mathrm{distinct}}}c_{Z_1}c_{Z_1Z_2}\cdots c_{Z_{k-1}Z_k}
\end{equation}
Now, for any $j:~0 \leq j \leq n^2-1$; consider a non-increasing sequence, $Z_1 \sptq \cdots \sptq Z_k$, in which $j+1$ of the $Z_i$'s are distinct. Then it has a strictly decreasing subsequence, $$Z_{i_1} \sptnq Z_{i_2} \sptnq \cdots \sptnq Z_{i_j} \sptnq Z_k.$$ So the non-increasing sequence, $Z_1 \sptq \cdots \sptq Z_k$, looks like this: \begin{equation}\label{Expn1}Z_1 = \cdots = Z_{i_1} \sptnq Z_{i_1 + 1} = \cdots = Z_{i_2} \sptnq \cdots = Z_{i_j} \sptnq Z_{i_j+1} = \cdots =Z_k.\end{equation} From Lemma~\ref{Obs}, we have: $c_{Z_1}c_{Z_1Z_2}\cdots c_{Z_{k-1}Z_k}$ is equal to $$c_{Z_1}q^{\dim(Z(Z_{i_1}))(i_1-1)}c_{Z_{i_1}Z_{i_2}}q^{\dim(Z(Z_{i_2}))(i_2-i_1-1)}\cdots c_{Z_{i_j}Z_k}q^{k-i_j-1}.$$

For $1 \leq u \leq j-1$, we have, $Z_{i_u} \sptnq Z_{i_{u+1}}$. Thus, $Z_{i_u} \sptnq Z_k$ for all $u:~1 \leq u\leq j$. Then by Lemma~\ref{L2}, we have $\dim(Z(Z_{i_u})) < \dim(Z(Z_k))$ for all $u:~1 \leq u \leq j$. Hence, for $1\leq u \leq j$, $$\dim(Z(Z_{i_u})) < m(n).$$ Therefore $$ \dim(Z(Z_{i_u})) \leq m(n) -1$$ for $1 \leq u \leq j$. Hence, $c_{Z_1}c_{Z_1Z_2}\cdots c_{Z_{k-1}Z_k}$ is bounded above by $$c_{Z_1}c_{Z_{i_1}Z_{i_2}}\cdots c_{Z_{i_j}Z_k}.q^{(m(n)-1)(i_j-j)}.q^{m(n)(k-i_j-1)},$$ which is bounded above by $$c_{Z_1}c_{Z_{i_1}Z_{i_2}}\cdots c_{Z_{i_j}Z_k}.q^{(m(n)-1)(i_j)}.q^{m(n)(k-i_j)}.$$ Now, as each of $c_{Z_1}, c_{Z_{i_1}Z_{i_2}}, \ldots, c_{Z_{i_j}Z_k}$ cannot be more than $q^{n^2}$, we have, \begin{equation*}\begin{aligned}c_{Z_1}c_{Z_1Z_2}\cdots c_{Z_{k-1}Z_k} &\leq q^{n^2(j+1)}.q^{[(m(n)-1)i_j+ m(n)(k-i_j)]} \\ &= q^{n^2(j+1)}.q^{(m(n)k-i_j)}
\end{aligned}
\end{equation*}
Here are some observations:
\begin{itemize}
\item We know that there are only a finite number of distinct algebras in $\Mnq$. Let that number be $f(n)$. For each $j$ as $0 \leq j \leq n^2-1$, there cannot be more than ${f(n) \choose j+1 }$ of them.
\item Given $Z_1 \sptq \cdots \sptq Z_k$, in which $j+1$ of them are distinct, i.e., there is a strongly decreasing subsequence of $Z_1 \sptq \cdots \sptq Z_k$: $$Z_{i_1} \sptnq Z_{i_2} \sptnq \cdots \sptnq Z_{i_j} \sptnq Z_k,$$  such that $Z_1 \sptq \cdots \sptq Z_k$, is as in Expression~\ref{Expn1}. Given this subset $S = \{i_1,\ldots, i_j\}$, at which the descents occur, $c_{Z_1}c_{Z_1Z_2}\cdots c_{Z_kZ_{k-1}}$ is bounded above by $$q^{n^2(j+1)}.q^{(m(n)k-\max(S))}.$$ But then this $S$ could be any size $j$ subset of $\{1,\ldots, k-1\}$. So, $c(n,k,q)$ is bounded above by $$\sum_{j=0}^{n^2-1}\left({f(n)\choose j+1}q^{n^2(j+1)}\sum_{\substack{S \sbteq \{1,\ldots,k-1\}\\|S| = j}}q^{(m(n)k-\max(S))}\right),$$
which is equal to $$\sum_{j=0}^{n^2-1}\left({f(n)\choose j+1}q^{n^2(j+1)}\sum_{r = j}^{k-1}\sum_{\substack{S \sbteq \{1,\ldots,k-1\}\\|S| = j\\ \Max(S) = r}}q^{(m(n)k-r)}\right).$$ But this is equal to $$\sum_{j=0}^{n^2-1}\left({f(n)\choose j+1}q^{n^2(j+1)}\sum_{r = j}^{k-1}{r-1 \choose j-1}q^{(m(n)k-r)}\right).$$ (Once $r$ is chosen, the remaining $j-1$ numbers are chosen from $1, \ldots, r-1$ in ${r-1 \choose j-1}$ ways.)\\

 Now, as ${r-1 \choose j-1} \leq r^j$, we get that \begin{equation*}\begin{aligned}c(n,k,q) &\leq q^{m(n)k}\sum_{j=0}^{n^2-1}\left({f(n)\choose j+1}q^{n^2(j+1)}\sum_{r = j}^{k-1}r^jq^{-r}\right)\\
&\leq q^{m(n)k}\sum_{j=0}^{n^2-1}\left({f(n)\choose j+1}q^{n^2(j+1)}\sum_{r = 0}^\infty r^jq^{-r}\right)
\end{aligned}
\end{equation*}

\end{itemize}

Now, for any fixed $j$, we can see by any of the routine tests (either the root or ratio test) that the series, $$\sum_{r = 0}^\infty r^jq^{-r},\text{ converges.}$$  So, let $$C_2 = \displaystyle\sum_{j=0}^{n^2-1}\left({f(n)\choose j+1}q^{n^2(j+1)}\displaystyle\sum_{r = 0}^\infty r^j q^{-r}\right),$$  then we have $$c(n,k,q) \leq s_2q^{m(n)k}.$$ So we have found positive constants $C_1$ and $C_2$ such that $$C_1q^{m(n)k} \leq c(n,k,q) \leq C_2 q^{m(n)k}$$ Hence $c(n,k,q)$, as a function of $k$ is asymptotically $q^{m(n)k}$ upto some constant factor.


\section{Asymptotic of Counting Tuples of Commuting Matrices}\label{S3}
In this section, instead of looking at simultaneous similarity classes of commuting tuples, we will look at the asymptotic of counting total number of tuples of commuting matrices. Let $C(n,k,q)$ denote the total number of $k$-tuples of commuting $n \times n$ matrices over $\F_q$ i.e., the size of $\Mnq^{(k)}$. Then we have,
\begin{equation}\label{E21}
C(n,k,q) = \sum_{Z \sbteq \Mnq} \frac{|\Gnq|}{|Z^*|}C_Z,
\end{equation}
 where $Z$ varies over conjugacy classes of subalgebras of $\Mnq$, $Z^*$ is the group of units of $Z$, and $C_Z$ is the total number of simultaneous similarity classes of $k$-tuples of commuting matrices whose common centralizer algebra is isomorphic to $Z$.\\

From the previous section, we see that $$C_Z = \sum_{\substack{Z_1 \sptq \cdots \sptq Z_k\\ Z_k = Z}} c_{Z_1}c_{Z_1Z_2}\cdots c_{Z_{k-1}Z_k},$$ where $Z_k = Z$. So we can rewrite equation~(\ref{E21}) as
\begin{equation}\label{E22}
C(n,k,q) = \sum_{Z_1 \sptq \cdots \sptq Z_k} \frac{|\Gnq|}{|Z^*_k|}c_{Z_1}c_{Z_1Z_2}\cdots c_{Z_{k-1}Z_k}.
\end{equation}
Equation~\ref{E22} is a modified version of equation~(\ref{Eqn2}).\\

Now, if we consider tuples, $(A_1,\ldots, A_k),$ whose first $l_0 ( = \left[\frac{n}{2}\right] +1)$ coordinates are as in examples~\ref{Eg1}~and~\ref{Eg2}, then we get $|Z_k| = q^{m(n)}$, and $|Z^*_k| = (q-1)q^{\left[\frac{n^2}{4}\right]}$. So we have $$\frac{|\Gnq|}{(q-1)q^{\left[\frac{n^2}{4}\right]}}c_{Z_1}c_{Z_1Z_2}\cdots c_{Z_{l_0}}q^{m(n)(k-l_0)} \leq C(n,k,q)$$ Thus, choose $$D_1 = \frac{|\Gnq|}{(q-1)q^{\left[\frac{n^2}{4}\right]}}c_{Z_1}c_{Z_1Z_2}\cdots c_{Z_{l_0}}q^{-m(n)l_0}.$$ Then we get $D_1q^{m(n)k} \leq C(n,k,q)$.\\

Now we can find an upper bound for $C(n,k,q)$. From equation~(\ref{E22}) we have $C(n,k,q)$ equal to $$\sum_{Z_1 \sptq \cdots \sptq Z_k} \frac{|\Gnq|}{|Z^*_k|}c_{Z_1}c_{Z_1Z_2}\cdots c_{Z_{k-1}Z_k.}$$ Now, as $\Gnq$ has only a finite number of subgroups, $\frac{|\Gnq|}{|Z^*_k|}$ is bounded above. Let that bound be $G(q)$. So we have \begin{equation*}\begin{aligned}C(n,k,q) &\leq G(q)\sum_{Z_1 \sptq \cdots \sptq Z_k}c_{Z_1}c_{Z_1Z_2}\cdots c_{Z_{k-1}Z_k}\\ &= G(q)c(n,k,q)\\ &\leq G(q)C_2q^{m(n)k} \text{ (From section~\ref{S2})}\end{aligned}
\end{equation*}
So let $D_2 =  G(q)C_2$, then we have $D_2 >0$ such that, $C(n,k,q) \leq D_2q^{m(n)k}$. This proves the theorem:
\begin{theorem}
The total number of $k$-tuples of commuting $n\times n$ matrices over $\F_q$: $C(n,k,q)$ is asymptotic to $q^{m(n)k}$ as a function of $k$.
\end{theorem}
Keeping $n$ and $q$ fixed, we could find the asymptotics of $c(n,k,q)$ and $C(n,k,q)$ as $k$ goes to $\infty$. We could instead keep $k$ and $q$ fixed and ask what are the asymptotics of $c(n,k,q)$ and $C(n,k,q)$ as $n$ goes to $\infty$. We could also keep $k$ and $n$ fixed and ask for the asymptotics of $c(n,k,q)$ and $C(n,k,q)$ as a function of $q$.
\subsection*{Acknowledgements} I thank my supervisor Prof.~Amritanshu~Prasad for the discussions we had about this topic and for feedback on the draft of this paper, and  Prof.~S.~Viswanath for a few suggestions while giving a talk on the results of this paper.

\bibliographystyle{alpha}
\bibliography{References}
\end{document}